\documentclass[a4paper,12pt]{amsart}

\usepackage[utf8]{inputenc}
\usepackage[T1]{fontenc}
\usepackage[english]{babel}
\usepackage{amsmath,amsthm,amsfonts}
\usepackage{graphicx}
\usepackage{xcolor}
\usepackage{csquotes}
\usepackage{subcaption}
\usepackage{hyphenat}
\usepackage{enumerate}
\usepackage{microtype}
\usepackage{csquotes}

\usepackage[maxbibnames=5, style = numeric, backend = biber, sorting = none]{biblatex}
\newtheorem{theorem}{Theorem}[section]

\newtheorem{lemma}{Lemma}[section]

\DeclareMathOperator{\ext}{ext}
\DeclareMathOperator{\sgn}{sgn}

\newcommand{\N}{\mathbb{N}}

\newcommand{\calA}{\mathcal{A}}
\newcommand{\tildeF}{\widetilde{F}}

\usepackage{titlesec}
\titleformat{\section}{\bf\Large}{\thesection. }{0em}{}[]

\begin{document}

\title{Plasticity of the unit ball of $c$ and $c_0$}
\author{Nikita Leo}
\address{Institute of Mathematics and Statistics, University of Tartu, Narva mnt 18, 51009 Tartu, Estonia}
\email{nikita.leo@ut.ee}
\date{}

\maketitle

\begin{abstract}
    \noindent We prove the plasticity of the unit ball of $c$. That is, we show that every non-expansive bijection from the unit ball of $c$ onto itself is an isometry. We also demonstrate a slightly weaker property for the unit ball of $c_0$ -- we prove that a non-expansive bijection is an isometry, provided that it has a continuous inverse.
\end{abstract}

\section{Introduction}

A map $F\colon X\to Y$ between two metric spaces is called \emph{non-expansive} if $d\big(F(a),F(b)\big)\leq d(a,b)$ for every $a$ and $b$ in $X$. If $d\big(F(a),F(b)\big) = d(a,b)$ for every $a$ and $b$ in $X$, then a function $F$ is called an \emph{isometry}. We call a metric space \emph{plastic} if every non-expansive bijection from the space onto itself is an isometry. The concept was introduced in \cite{NPW2006} by S.~A.~Naimpally, Z.~Piotrowski and E.~J.~Wingler. It seems that the class of plastic metric spaces does not have a simple characterization. The only general result is that every totally bounded space is plastic. Conversely, it is known that a plastic metric space need not be totally bounded nor bounded. It can also be shown that a bounded space need not be plastic. These results were obtained in \cite{NPW2006}.

It is an open question whether the unit ball of every Banach space is a plastic metric space. The question was posed in 2016 by B.~Cascales, V.~Kadets, J.~Orihuela and E.~J.~Wingler \cite{CKOW2016}. In this paper, they also demonstrated the plasticity of the unit ball of strictly convex Banach spaces. The unit ball of a finite-dimensional space is compact and compactness implies plasticity, so the question is really just about the infinite-dimensional spaces. The situation is different there, because an infinite-dimensional space can contain subsets that are very similar to the unit ball, but are not plastic. The simplest example is an ellipsoid with suitably chosen lengths of axes \cite[Example 2.7]{CKOW2016}. In 2016, the plasticity of the unit ball was also proved for the space $\ell_1$ and the proof was presented by V.~Kadets and O.~Zavarzina \cite{KZ2016}. In 2018, they generalised this result to $\ell_1$-sums of strictly convex spaces \cite{KZ2018}. Finally, the same year C.~Angosto, V.~Kadets and O.~Zavarzina demonstrated the plasticity of the unit ball for spaces whose unit sphere is the union of all its finite-dimensional polyhedral extreme subsets \cite{AKZ2018}. These are all the positive results that were obtained so far. The class of strictly convex spaces is a subclass of the other two classes, so the last two results might be seen as generalizations of the first.

The plasticity of the unit ball is known for all $\ell_p$ sequence spaces except $\ell_\infty$ -- there is a positive result for $\ell_1$, while for $p$ in $(1,\infty)$ the space $\ell_p$ is strictly convex and there is a positive result for strictly convex spaces. The difficulty of $\ell_\infty$ seems to be lying in the fact that this space is too big. In particular, it is not separable, while the remaining $\ell_p$ spaces are. However, $\ell_\infty$ contains some common separable spaces like $c$ and $c_0$. While the plasticity of the unit ball of $\ell_\infty$ seems to be a hard problem to tackle, spaces $c$ and $c_0$ may be worth a try. In this paper, we are going to prove the plasticity of the unit ball of $c$, extending the list of positive results. We are also going to establish a slightly weaker property for the unit ball of $c_0$ -- we show that a non-expansive bijection is an isometry, provided that it has a continuous inverse.

Extreme points play a significant role in the problem of plasticity of the unit ball. All the positive results obtained so far are for spaces that have many extreme points. The first positive result was for strictly convex spaces and these are spaces where all the points of the unit sphere are extreme. The later extensions allow for some non-extreme points, but they still require the space to have many extreme points. On the other hand, nothing is known about spaces with little or no extreme points at all. It is natural to suppose that some of these spaces may actually not have a plastic unit ball. This provides motivation to study the plasticity of the unit ball in spaces with little or no extreme points. A positive result for any such space would be also significant -- it would show that the plasticity of the unit ball does not require the space to have many extreme points. One of such spaces is $c_0$, as it has no extreme points. It is also one of the two spaces considered in this paper. In contrast, the space $c$ can be said to have many extreme points.
\newpage
\section{Preliminaries and notation}

Let us list the preliminaries. The next theorem describes the behaviour of a non-expansive bijection from the unit ball of a normed space to itself. It lists some of the main tools used in the study of the problem. These properties were first observed in \cite{CKOW2016}, the first article on plasticity of the unit ball, where the positive result for strictly convex spaces was obtained. As mentioned before, extreme points are essential in the context at hand. The reason for this is the last item of the following theorem.

\begin{theorem}[{\cite[Theorem 2.3]{CKOW2016}}]\label{BnEproperties}
    Let $X$ be a normed space and let $F\colon B_X\to B_X$ be a non-expansive bijection. Then
        \begin{enumerate}[1)]
            \item $F(0)=0$;
            \item if $x\in S_X$, then $F^{-1}(x)\in S_X$;
            \item if $x\in\ext B_X$, then $F^{-1}(x)\in \ext B_X$ and $F^{-1}(\alpha x)=\alpha F^{-1}(x)$ for each $\alpha\in[-1,1]$.
        \end{enumerate}
\end{theorem}

When dealing with plasticity of the unit ball, it is also useful to know the following result by P.~Mankiewicz.

\begin{theorem}[{\cite[Theorem 5]{Mankiewicz1972}}]\label{Mankiewicz}
    Let $X$ and $Y$ be normed spaces and let $U$ be a subset of $X$ and $V$ be a subset of $Y$. If $U$ and $V$ are convex with non-empty interior and there exists an isometric bijection $F\colon U\to V$, then $F$ extends to an affine isometric bijection $\widetilde{F}\colon X\to Y$.
\end{theorem}
The theorem implies that if $X$ is a normed space and $F\colon B_X\to B_X$ is an isometric bijection, then $F$ extends to an isometric automorphism of $X$ -- a linear isometric bijection from $X$ onto itself. This can be applied when proving the plasticity of the unit ball, because at some point of the proof one might discover that the unit ball of some finite-dimensional subspace is mapped bijectively onto itself or a copy of itself. In this situation, one can apply the fact that the unit ball of a finite-dimensional space is plastic and then apply the mentioned result to conclude the linearity, which can turn useful in the future.

The last theorem is also valuable in one other way. If we want to prove that every non-expansive bijection on the unit ball of some space is an isometry, then it might be useful to know what kinds of isometries are there. The last result says that these are precisely the restrictions of the isometric automorphisms of the space. One possible approach to proving the plasticity of the unit ball in some specific space $X$ consists of considering a non-expansive bijection $F\colon B_X\to B_X$, retrieving some information about this function to choose an isometric automorphism and then proving that the two functions actually coincide. This is also the approach used in the two proofs presented in this paper. This approach requires that we know the characterization of the isometric automorphisms of $X$.

Let us fix the notation for the following two sections. Given a sequence $x\in c$, denote the $n$-th element by $x_n$. Given a sequence $\xi\in c^\N$, denote the $k$-th sequence by $\xi^k$ and the $n$-th element of the $k$-th sequence by $\xi^k_n$. For $n\in\N$ denote by $e^n$ a sequence such that $e^n_n=1$ and $e^n_i=0$ for every $i\in\N\setminus\{n\}$. Denote the unit ball of $c$ by $B$. For $h\in[-1,1]$ denote by $B_h$ the subset of $B$ which consists of all sequences converging to $h$. In particular, $B_0$ is going to stand for the unit ball of $c_0$. For a subset $S\subset \N$ and $h\in[-1,1]$ denote by $B_h^S$ the subset of $B_h$ defined by
\[B_h^S=\{x\in B_h: \text{$x_n=h$ for all $n\in \N\setminus S$}\}.\]
For $h\in[-1,1]$ denote by $B_h^*$ the subset of $B_h$ defined by
\[B_h^*=\{x\in B_h\colon \text{$x_n\ne h$ for finitely many $n\in\N$}\}.\]
For $x\in c$ and $r>0$ denote by $B(x,r)$ the corresponding closed ball of the space $c$.

As pointed above, it is good to know the characterization of the isometric automorphisms. The isometric automorphisms of $c_0$ have the form $\calA(x)_{\sigma_n}=\alpha_n x_n$, where $\sigma\colon \N\to\N$ is a bijection and $\alpha$ is a sequence of ones and minus ones. The isometric automorphisms of $c$ are the same, except that the sequence $\alpha$ should be constant from some point. When dealing with plasticity of the unit ball, it is also important to know the extreme points. While the space $c_0$ has no extreme points, the extreme points of $c$ are the sequences that consist of just ones and minus ones and are constant from some point.

\newpage
\section{The space $c_0$}

Let us consider an arbitrary non-expansive bijection $F\colon B_0\to B_0$ from the unit ball of $c_0$ onto itself. We are going to try to infer as much as possible about the behaviour of this function. The maximum goal is to show that $F$ is an isometry, but we will not be able to achieve this. However, we will show that $F$ is an isometry, provided that $F^{-1}$ is continuous. The first step is to extract some information about $F$ to choose an isometric automorphism of $c_0$ that the function $F$ seems to resemble. We need the following lemma, which is concerned with covering the unit ball by two closed balls of radius one.

\begin{lemma}\label{c0_step-1}
    Let $x$ and $y$ be two non-zero elements of $B_0$. The balls $B(x,1)$ and $B(y,1)$ cover the ball $B_0$ if and only if there exists an index $n$ such that $x_i=y_i=0$ for all $i\ne n$ and either $x_n$ is positive and $y_n$ is negative or $x_n$ is negative and $y_n$ is positive.
\end{lemma}
\begin{proof}
    Suppose that there exists an index $n$ such that $x_i=y_i=0$ for all $i\ne n$ and either $x_n$ is positive and $y_n$ is negative or $x_n$ is negative and $y_n$ is positive. Consider the case where $x_n$ is positive and $y_n$ is negative. Let $z$ be an arbitrary element of $B_0$. We see that if $z_n\geq 0$, then $z\in B(x,1)$, and if $z_n\leq 0$, then $z\in B(y,1)$. That is, $z$ is always contained in at least one of the balls $B(x,1)$ and $B(y,1)$. This means that the balls $B(x,1)$ and $B(y,1)$ cover the ball $B_0$. The second case is analogous. This proves one of the two directions.
    
    For the second direction, assume that $B_0$ is covered by $B(x,1)$ and $B(y,1)$. Since $x$ is non-zero, there exists an index $n$ such that $x_n\ne 0$. Suppose that there exists an index $i\ne n$ such that $y_i\ne 0$. Consider a sequence $z=-\sgn(x_n)e^n-\sgn(y_i)e^i$. Note that $z\in B_0$, but $z$ is not covered by $B(x,1)$ and $B(y,1)$. This contradicts our assumption, so we can conclude that $y_i=0$ for every $i\ne n$. If $y_n$ was also equal to zero, then $y$ would be equal to zero, therefore $y_n\ne 0$. Now we can repeat the above argument swapping the roles of $x$ and $y$. As a result, we get that $x_i=0$ for every $i\ne n$. Finally, we have to exclude the possibility that $x_n$ and $y_n$ are both positive or both negative. This is easy to see, because if $x_n$ and $y_n$ are both positive, then $B(x,1)$ and $B(y,1)$ do not cover $-e^n\in B_0$, and if $x_n$ and $y_n$ are both negative, then $B(x,1)$ and $B(y,1)$ do not cover $e^n\in B_0$. This proves the second direction.
\end{proof}

Now we can extract some information about $F$ to choose a corresponding isometric automorphism of $c_0$.

\begin{lemma}\label{c0_step0}
    There exists a bijection $\sigma\colon\mathbb{N}\to\mathbb{N}$ and a sequence $\alpha\colon\mathbb{N}\to\{-1,1\}$ such that for every $x\in B_0$ and $n\in\mathbb{N}$ we have the following:
    \begin{enumerate}[1)]
        \item if $x_n=0$, then $F(x)_{\sigma_n}=0$;
        \item if $x_n<0$ and $\alpha_n=1$ ($\alpha_n=-1$), then $F(x)_{\sigma_n}\leq0$ ($F(x)_{\sigma_n}\geq0$);
        \item if $x_n>0$ and $\alpha_n=1$ ($\alpha_n=-1$), then $F(x)_{\sigma_n}\geq0$ ($F(x)_{\sigma_n}\leq0$).
    \end{enumerate}
\end{lemma}
\begin{proof}
    Let $n\in\mathbb{N}$ be arbitrary. Consider elements $e^n$ and $-e^n$. Note that $B_0$ is covered by the balls $B(e^n,1)$ and $B(-e^n,1)$. As $F$ is non-expansive and surjective, then $B_0$ should be also covered by the balls $B\big(F(e^n),1\big)$ and $B\big(F(-e^n),1\big)$. Moreover, item 1) of Theorem \ref{BnEproperties} implies that $F(e^n)$ and $F(-e^n)$ are non-zero. The previous lemma implies that there exists an index $k$ such that $F(e^n)_i=F(-e^n)_i=0$ for all $i\ne k$ and either $F(e^n)_k$ is positive and $F(-e^n)_k$ is negative or $F(e^n)_k$ is negative and $F(-e^n)_k$ is positive. Define $\sigma_n=k$. Define $\alpha_n=1$ if $F(e^n)_k$ is positive and $\alpha_n=-1$ otherwise.
    
    Now, we can show that if a sequence $x\in B_0$ is such that $x_i=0$ for every $i\ne n$, then $F(x)_i=0$ for all $i\ne k$. Let $x\in B_0$ be such that $x_i=0$ for every $i\ne n$. If $x_n=0$, then $x$ is zero, so $F(x)$ is also zero by item 1) of Theorem \ref{BnEproperties} and it is true that $F(x)_i=0$ for all $i\ne k$. So let us consider the case where $x_n\ne 0$. Suppose that $x_n>0$. Note that $B_0$ is covered by the balls $B(-e^n,1)$ and $B(x,1)$. As $F$ is non-expansive and surjective, then $B_0$ should be also covered by the balls $B\big(F(-e^n),1\big)$ and $B\big(F(x),1\big)$. Item 1) of Theorem \ref{BnEproperties} implies that $F(-e^n)$ and $F(x)$ are non-zero, so the previous lemma implies that $F(x)_i=0$ for every $i\ne k$. The case $x_n<0$ is analogous.
    
    Consider a function $f_n\colon[-1,1]\to[-1,1]$ defined by $f_n(t)=F(te^n)_{\sigma_n}$. The non-expansiveness and the injectivity of $F$ imply the same properties of $f_n$. Since $f_n$ is continuous and injective, then it is either strictly increasing or strictly decreasing. If $\alpha_n=1$, then $f_n$ is increasing, $f_n(-1)<0$, $f_n(0)=0$ and $f_n(1)>0$. If $\alpha_n=-1$, then $f_n$ is decreasing, $f_n(-1)>0$, $f_n(0)=0$ and $f_n(1)<0$. Since $f_n$ is continuous, then the image of $f_n$ is a segment. This segment should have the form $[a,b]$, where $a<0$ and $b>0$. Now it is easy to see that the injectivity of $F$ requires the injectivity of $\sigma$. Indeed, suppose by contrary that there exist distinct indices $m$ and $n$ such that $\sigma_m=\sigma_n=k$. Then the line segments $[-e^n,e^n]$ and $[-e^m,e^m]$ have just one common point, but their images are line segments $[ae^k,be^k]$ and $[a'e^k,b'e^k]$, where $a,a'<0$ and $b,b'>0$. These segments have more than one common point, which contradicts the injectivity of $F$.
    
    We want to show that for each $x\in B_0$ and $n\in\N$ the items 1)--3) are true. As $F$ is continuous and $B_0^*$ is dense in $B_0$, then it suffices to consider the case where $x\in B_0^*$. Denote by $S$ the set $\{n\in\N\colon x_n\ne 0\}$. Note that $S$ is finite. For every $n\in S$ let $s^n=-\sgn(x_n)e^n$. Note that $B_0$ is covered by the ball $B(x,1)$ together with the balls $B(s^n,1)$, $n\in S$. Indeed, for arbitrary $z\in B_0$, if there exists an index $n\in S$ such that either $x_n>0$ and $z_n<0$ or $x_n<0$ and $z_n>0$, then $z$ belongs to $B(s^n,1)$. Otherwise, $z$ belongs to $B(x,1)$. It follows that $B_0$ should be also covered by the balls $B\big(F(x),1\big)$ and $B\big(F(s^n),1\big)$, $n\in S$.
    
    Now, let us construct a sequence $y$ as follows:
    \begin{itemize}
        \item for $k\not\in \sigma(S)$ let $y_k=0$;
        \item for $n\in S$ let $y_{\sigma_n}=\sgn(x_n)\alpha_n$.
    \end{itemize}
    
    Note that $y_k=0$ for $k\not\in \sigma(S)$ and $y_k=1$ or $y_k=-1$ for $k\in \sigma(S)$. As $\sigma(S)$ is finite, then $y$ belongs to $B_0$. Note that $y$ does not belong to any of the balls $B\big(F(s^n),1\big)$, $n\in S$. On the other hand, $y$ is an element of $B_0$ and $B_0$ should be covered by the balls $B\big(F(x),1\big)$ and $B\big(F(s^n),1\big)$, $n\in S$. It follows that $y$ should belong to $B\big(F(x),1\big)$. This implies items 2) and 3).
    
    To prove item 1), consider an index $k\not\in \sigma(S)$. Alter the definition of $y$ such that $y_k=1$. Note that $y$ still belongs to $B_0$ and does not belong to any of the balls $B\big(F(s^n),1\big)$, $n\in S$, so it should belong to $B\big(F(x),1\big)$. This yields $F(x)_k\geq 0$. Similarly, if we take $y_k=-1$, then we get $F(x)_k\leq 0$. So we can conclude that $F(x)_k=0$. This proves item 1).
    
    Continuity of $F$, together with the fact that $B_0^*$ is dense in $B_0$, allows us to extend the result to all elements of $B_0$. Finally, we have to make sure $\sigma$ is a surjection. Suppose by contrary that there exists an index $k\not\in\sigma(\N)$. From the above argument it follows that $F(x)_k=0$ for every $x\in B_0^*$. Continuity of $F$ and the fact that $B_0^*$ is dense in $B_0$ imply that the same holds for every $x\in B_0$. This contradicts the surjectivity of $F$. Therefore, $\sigma$ should be surjective.
\end{proof}

The previous lemma fixes a bijection $\sigma\colon\mathbb{N}\to\mathbb{N}$ and a sequence $\alpha\colon\mathbb{N}\to\{-1,1\}$. Let $\calA$ be the corresponding isometric automorphism of $c_0$ defined by $\calA(x)_{\sigma_n}=\alpha_n x_n$. Denote by $G$ the restriction of $\calA$ to $B_0$, which is an isometric bijection from $B_0$ onto itself. Our goal is to show that $F=G$. This is equivalent to showing that $G^{-1}\circ F$ is the identity map of $B_0$. Denote the map $G^{-1}\circ F$ by $\tildeF$. Note that the definition of $\tildeF$ implies that $\tildeF$ is a non-expansive bijection from $B_0$ onto itself. The previous lemma implies that $\tildeF$ has the following properties.

\begin{lemma}\label{c0_step1}
    For each $x\in B_0$ and $n\in\N$,
    \begin{enumerate}[1)]
        \item if $x_n=0$, then $\tildeF(x)_n=0$;
        \item if $x_n<0$, then $\tildeF(x)_n\leq0$;
        \item if $x_n>0$, then $\tildeF(x)_n\geq0$.
    \end{enumerate}
\end{lemma}

The three properties listed above imply the following property of $\tildeF^{-1}$.

\begin{lemma}\label{c0_step1a}
    For each $y\in B_0$ and $n\in\N$, if $y_n<0$, then $\tildeF^{-1}(y)_n<0$, and if $y_n>0$, then $\tildeF^{-1}(y)_n>0$.
\end{lemma}

Now we are going to continue collecting some properties that describe the behaviour of $\tildeF$.

\begin{lemma}\label{c0_step2}
    For each $y\in B_0$ and $n\in\N$, if $y_n=1$ and $y_i\ne 0$ for all $i\ne n$, then $\tildeF^{-1}(y)_n=1$, and if $y_n=-1$ and $y_i\ne 0$ for all $i\ne n$, then $\tildeF^{-1}(y)_n=-1$.
\end{lemma}
\begin{proof}
    Let $y\in B_0$ be such that $y_n=1$ and $y_i\ne 0$ for all $i\ne n$. Consider a sequence $z\in B_0$ defined by $z_n=-1$ and $z_i=y_i$ for all $i\ne n$. Denote $\tildeF^{-1}(y)$ and $\tildeF^{-1}(z)$ by $y'$ and $z'$. The distance between $y$ and $z$ equals two, so the distance between $y'$ and $z'$ should be also equal to two. Consider an index $i\ne n$. We know that $y_i\ne 0$, so $y_i$ is either positive or negative. If $y_i>0$, then Lemma \ref{c0_step1a} implies that $y'_i$ and $z'_i$ should be positive too. Similarly, if $y_i<0$, then Lemma \ref{c0_step1a} implies that $y'_i$ and $z'_i$ should be also negative. In either case we have $|y'_i-z'_i|<1$. It follows, that for the distance between $y'$ and $z'$ to be equal to two, we need to have $|y'_n-z'_n|=2$. This means that either $y'_n=1$ and $z'_n=-1$ or $y'_n=-1$ and $z'_n=1$. However, Lemma \ref{c0_step1a} excludes the second case. Therefore, we have $y'_n=1$ as wanted. The case $y_n=-1$ is similar.
\end{proof}

\begin{lemma}\label{c0_step3}
    For each $x\in B_0$ and $n\in\mathbb{N}$, if $x_n<0$, then $\tildeF(x)_n\in[x_n,0]$, and if $x_n>0$, then $\tildeF(x)_n\in[0,x_n]$.
\end{lemma}
\begin{proof}
    Let us consider the case $x_n>0$. Lemma \ref{c0_step1} implies that $\tildeF(x)_n\geq 0$, so it remains to show that $\tildeF(x)_n\leq x_n$. Let us construct a sequence $y\in B_0$ as follows. First, let $y_n=-1$. Then, for every $i\ne n$
    \begin{itemize}
        \item choose $y_i$ from $[-1,0)$ if $x_i<0$;
        \item choose $y_i$ from $(0,1]$ if $x_i>0$;
        \item choose $y_i$ from $[-1,1]\setminus\{0\}$ if $x_i=0$.
    \end{itemize}
    For $y$ to belong to $B_0$, it is important to choose $y_i$ such that the sequence converges to zero. Clearly, such choice is possible. One possible choice is to define $y_i=1/i$ for $x_i\geq 0$ and $y_i=-1/i$ for $x_i<0$.
    
    Denote $\tildeF^{-1}(y)$ by $y'$. Note that $y_n=-1$ and $y_i\ne 0$ for every $i\ne n$. Lemma \ref{c0_step2} implies that $y'_n=-1$. Let us show that $|x_i-y'_i|\leq 1$ for every $i\ne n$. Consider an index $i$ distinct from $n$. If $x_i<0$, then $y_i<0$, so Lemma \ref{c0_step1a} implies that $y'_i<0$. As $x_i$ and $y'_i$ are both negative, then $|x_i-y'_i|<1$. Similarly, if $x_i>0$, then $y_i>0$, so Lemma \ref{c0_step1a} implies that $y'_i>0$. As $x_i$ and $y'_i$ are both positive, then $|x_i-y'_i|<1$. Finally, if $x_i=0$, then the inequality $|x_i-y'_i|\leq 1$ is obvious. Now we see that the distance between $x$ and $y'$ is equal to $1+x_n$. Indeed, $|x_n-y'_n|=1+x_n$ and $|x_i-y'_i|\leq 1$ for $i\ne n$. This implies that the distance between $\tildeF(x)$ and $y$ is at most $1+x_n$, therefore $|\tildeF(x)_n-y_n|\leq 1+x_n$, which yields $\tildeF(x)_n\leq x_n$. This concludes the proof for the case $x_n>0$. The proof for the case $x_n<0$ is similar.
\end{proof}

The last lemma implies the following properties of the inverse function.

\begin{lemma}\label{c0_step3a}
    For each $y\in B_0$ and $n\in\N$, if $y_n<0$, then $\tildeF^{-1}(y)_n\in[-1,y_n]$, and if $y_n>0$, then $\tildeF^{-1}(y)_n\in[y_n,1]$.
\end{lemma}

The important feature of $c_0$ is that every element attains its norm. That is, for every $x\in c_0$ there is an index $n$ such that $|x_n|=\|x\|$. In other words, the supremum used to define the norm of $x$ is actually a maximum. We are going to make use of this feature in the proof of the next proposition.

\begin{lemma}\label{c0_step4}
    Let $y\in B_0$ be such that the set $S=\{i\in\N\colon y_i=0\}$ is finite. Then $\tildeF^{-1}(y)_i=0$ for each $i\in S$.
\end{lemma}
\begin{proof}
    Let us proceed by induction on the number of elements of $S$. If $S$ is empty, then there is nothing to prove. Now, let $N$ be a non-negative integer and suppose that the claim holds whenever $S$ has up to $N$ elements. Let us show that the claim also holds when $S$ has $N+1$ elements. So let $y\in B_0$ be a sequence such that the set $S$ has $N+1$ elements. Denote $\tildeF^{-1}(y)$ by $y'$. We are going to use a proof by contradiction. Suppose by contrary that there exists $n\in S$ such that $y'_n\ne 0$. Construct a sequence $z\in B_0$ as follows. First, set $z_n=1$ if $y'_n>0$ and $z_n=-1$ if $y'_n<0$. Then, for each $i\ne n$
    \begin{itemize}
        \item choose $z_i$ from $[-1,0)$ if $y'_i<0$;
        \item choose $z_i$ from $(0,1]$ if $y'_i>0$;
        \item set $z_i=0$ if $y'_i=0$.
    \end{itemize}
    For $z$ to belong to $B_0$, it is important to choose $z_i$ such that the sequence converges to zero. Clearly, such choice is possible. One possible choice is to define $z_i=1/i$ for $y'_i>0$ and $z_i=-1/i$ for $y'_i<0$.
    
    By definition, $z_n\ne 0$. If $i\in\N\setminus S$, then $y_i\ne 0$, so Lemma \ref{c0_step1a} implies that $y_i'\ne 0$ and the definition of $z$ implies that $z_i\ne 0$. It follows that the set $\{i\in\N\colon z_i=0\}$ is contained in the set $S\setminus\{n\}$. Therefore, the set $\{i\in\N\colon z_i=0\}$ has at most $N$ elements. This means that the induction hypothesis can be applied to $z$, so we know that $\tildeF^{-1}(z)_i=0$ whenever $z_i=0$.
    
    Denote $\tildeF^{-1}(z)$ by $z'$. Let us show that the distance between $y'$ and $z'$ is smaller than one. Since every element of $c_0$ attains its norm, it suffices to show that $|y'_i-z'_i|<1$ for each $i\in\N$. Let $i\in\N$ be arbitrary. Consider the case $y'_i>0$. The definition of $z$ implies that $z_i>0$. As $z_i>0$, then Lemma \ref{c0_step1a} implies that $z'_i>0$. Since $z'_i$ and $y'_i$ are both positive, then $|y'_i-z'_i|<1$. Similarly, if $y'_i<0$, then the definition of $z$ implies that $z_i<0$. As $z_i<0$, then Lemma \ref{c0_step1a} implies that $z'_i<0$. Since $z'_i$ and $y'_i$ are both negative, then $|y'_i-z'_i|<1$. Finally, if $y'_i=0$, then the definition of $z$ implies that $z_i=0$. Since $z_i=0$, then the application of the induction hypothesis to $z$ gives $z'_i=0$, which makes the inequality $|y'_i-z'_i|<1$ obvious. This shows that the distance between $y'$ and $z'$ is smaller than one.
    
    If the distance between $y'$ and $z'$ is smaller than one, then the distance between $y$ and $z$ should be also smaller than one, but we have $|y_n-z_n|=1$, which is a contradiction.
\end{proof}

\begin{lemma}\label{c0_step5}
    Let $x\in B_0$ be such that the set $S=\{i\in \N\colon x_i=0\}$ is finite. Then $\tildeF(x)_i\ne 0$ for each $i\in \N\setminus S$.
\end{lemma}
\begin{proof}
    Let $n\in\N\setminus S$ be arbitrary. Construct a sequence $y\in B_0$ as follows. First, set $y_n=\sgn(x_n)$. Then, for each $i\ne n$
    \begin{itemize}
        \item choose $y_i$ from $[-1,0)$ if $x_i<0$;
        \item choose $y_i$ from $(0,1]$ if $x_i>0$;
        \item set $y_i=0$ if $x_i=0$.
    \end{itemize}
    For $y$ to belong to $B_0$, it is important to choose $y_i$ such that the sequence converges to zero. Clearly, such choice is possible. Denote $\tildeF^{-1}(y)$ by $y'$. Note that the sequence $y$ has only finitely many zeros. Therefore, Lemma \ref{c0_step4} can be applied to $y$. This means that $y'_i=0$ for every $i\in S$. Now, we can show that the distance between $x$ and $y'$ is smaller than one -- the argument is identical to the one that appeared in the proof of the previous lemma. Since the distance between $x$ and $y'$ is smaller than one, then the distance between $\tildeF(x)$ and $y$ should be also smaller than one. In particular, we should have $|\tildeF(x)_n-y_n|<1$, which yields $\tildeF(x)_n\ne 0$.
\end{proof}

Lemma \ref{c0_step4} says that if we have $y_n=0$, then we should also have $\tildeF^{-1}(y)_n=0$, provided that the sequence $y$ has only finitely many zeros. If we assume the continuity of $\tildeF^{-1}$, then we can get rid of that additional assumption.

\begin{lemma}\label{c0_step6}
    Suppose that $\tildeF^{-1}$ is continuous. Let $y\in B_0$ and $n\in\N$ be such that $y_n=0$. Then $\tildeF^{-1}(y)_n=0$.
\end{lemma}
\begin{proof}
    It is possible to construct a sequence $\psi\colon \N\to B_0$ such that $\psi^k_n=0$, $\psi^k_i\ne 0$ for $i\ne n$ and $\psi^k\to y$. One possible choice is to define
    
    \[
        \psi^k_i=
        \begin{cases}
            0,   & i=n,\\
            y_i, & i\ne n,\ y_i\ne 0,\\
            1/(ik), & i\ne n,\ y_i=0.
        \end{cases}
    \]
    
    For each $k\in\N$, the sequence $\psi^k\in B_0$ has exactly one zero at index $n$, so Lemma \ref{c0_step4} implies that $\tildeF^{-1}(\psi^k)_n=0$. Since $\psi^k\to y$ and $\tildeF^{-1}$ is continuous, then $\tildeF^{-1}(\psi^k)\to \tildeF^{-1}(y)$. This implies the convergence $\tildeF^{-1}(\psi^k)_n\to \tildeF^{-1}(y)_n$. As $\tildeF^{-1}(\psi^k)_n=0$ for every $k\in\N$, then it follows that $\tildeF^{-1}(y)_n=0$.
\end{proof}

Now we can prove the main result.

\begin{theorem}\label{c0_main_result}
    If $F^{-1}$ is continuous, then $F$ is an isometry.
\end{theorem}
\begin{proof}
    Let $S$ be a finite subset of $\N$. Consider the subset 
    \[B_0^S=\{x\in B_0: \text{$x_n=0$ for all $n\in \N\setminus S$}\}.\]
    Lemma \ref{c0_step1} says that $\tildeF(x)_n=0$ whenever $x_n=0$. This yileds the inclusion $\tildeF(B^S_0)\subset B^S_0$. If $F^{-1}$ is continuous, then $\tildeF^{-1}$ is also continuous, so we can apply Lemma \ref{c0_step6}, which yields the inclusion $\tildeF^{-1}(B^S_0)\subset B^S_0$. Combining these two together, we get that the set $B^S_0$ is mapped bijectively onto itself. It follows that the restriction of $\tildeF$ to $B^S_0$ is a non-expansive bijection from the unit ball of a finite-dimensional space onto itself. Since the unit ball of a finite-dimensional space is plastic, then it follows that the restriction of $\tildeF$ to $B^S_0$ is an isometry. Moreover, Theorem \ref{Mankiewicz} says that the latter is actually a restriction of an isometric automorphism of the underlying finite-dimensional space. Combining this with some previously acquired information, we can conclude that the restriction of $\tildeF$ to $B^S_0$ is an identity map. Indeed, Lemma \ref{c0_step1} says that for each $n\in S$ the element $e^n$ is mapped to $te^n$, where $t>0$. However, for the norm to be preserved, we need to have $t=1$. Therefore, the elements $e^n$, $n\in S$ are mapped to itself. Since every element of $B^S_0$ is a linear combination of these, then the linearity implies that $\tildeF$ should keep all elements of $B^S_0$ in place.
    
    If $\tildeF$ restricted to $B^S_0$ is an identity map, then the restriction of $\tildeF$ to $B^*_0$ is also an identity map, because the latter is the union of all the subsets $B^S_0$, where $S$ is a finite subset of $\N$. Since $\tildeF$ restricted to $B^*_0$ is an identity map, $\tildeF$ is continuous and $B^*_0$ is dense in $B_0$, then it follows that $\tildeF$ is an identity map of $B_0$. This means that $F$ is a restriction of an isometric automorphism of $c_0$ defined by $\calA(x)_{\sigma_n}=\alpha_n x_n$. In particular, $F$ is an isometry.
\end{proof}
\newpage
\section{The space $c$}

Let us consider an arbitrary non-expansive bijection $F\colon B\to B$ from the unit ball of $c$ onto itself. Our goal is to show that $F$ is an isometry. Some parts of the proof will be identical to the corresponding parts of the proof for $c_0$ and we are going to omit these parts. Therefore, it is advisable to take a look at the previous section before reading the proof at hand.

As with the space $c_0$, the first step is to extract some information about $F$ to choose a corresponding isometric automorphism of $c$. The next lemma is an analog of Lemma \ref{c0_step-1}. The only difference is that now we have $B$ instead of $B_0$. The proof is identical to the one of Lemma \ref{c0_step-1}.

\begin{lemma}
    Let $x$ and $y$ be two non-zero elements of $B$. The balls $B(x,1)$ and $B(y,1)$ cover the ball $B$ if and only if there exists an index $n$ such that $x_i=y_i=0$ for all $i\ne n$ and either $x_n$ is positive and $y_n$ is negative or $x_n$ is negative and $y_n$ is positive.
\end{lemma}

Now we can retrieve some information about $F$ to fix an isometric automorphism of $c$ that the function $F$ seems to resemble. The next lemma is an analog of Lemma \ref{c0_step0}. The difference from the space $c_0$ is that now we have to ensure that the sequence $\alpha$ is constant from some point.

\begin{lemma}\label{c_step0}
    There exists a bijection $\sigma\colon\mathbb{N}\to\mathbb{N}$ and a sequence $\alpha\colon\mathbb{N}\to\{-1,1\}$, which is constant starting from some index, such that for every $x\in B$ and $n\in\mathbb{N}$ we have the following:
    \begin{enumerate}[1)]
        \item if $x_n=0$, then $F(x)_{\sigma_n}=0$;
        \item if $x_n<0$ and $\alpha_n=1$ ($\alpha_n=-1$), then $F(x)_{\sigma_n}\leq0$ ($F(x)_{\sigma_n}\geq0$);
        \item if $x_n>0$ and $\alpha_n=1$ ($\alpha_n=-1$), then $F(x)_{\sigma_n}\geq0$ ($F(x)_{\sigma_n}\leq0$).
    \end{enumerate}
\end{lemma}
\begin{proof}
    The first part of the proof repeats the first three paragraphs of the proof of Lemma \ref{c0_step0}. We only have to replace $B_0$ by $B$. We fix a function $\sigma\colon\N\to\N$ and a sequence $\alpha\colon\N\to\{-1,1\}$. We show that for each $n\in\N$ there exists a continuous function $f_n\colon[-1,1]\to[-1,1]$ such that $F(t e^n)=f_n(t)e^{\sigma_n}$ for each $t\in[-1,1]$. We also know that if $\alpha_n=1$, then $f_n$ is strictly increasing, $f_n(-1)<0$, $f_n(0)=0$ and $f_n(1)>0$, and if $\alpha_n=-1$, then $f_n$ is strictly decreasing, $f_n(-1)>0$, $f_n(0)=0$ and $f_n(1)<0$. We also show that the function $\sigma$ is injective.
    
    The next step is to ensure that the sequence $\alpha$ is constant from some point. Let $y$ be an arbitrary extreme point of $B$. The sequence $y$ consists of ones and minus ones and is constant from some point. Denote $F^{-1}(y)$ by $x$. By item 3) of Theorem \ref{BnEproperties} we know that $x$ is also an extreme point. Therefore, the sequence $x$ consists of ones and minus ones and is constant from some point. Let us show that for each $n\in\N$ we have $y_{\sigma_n}=\alpha_n x_n$. Let $n\in\N$ be arbitrary. We know that $x_n$ is either $1$ or $-1$. Let us consider the case $x_n=1$. We need to show that $y_{\sigma_n}=\alpha_n$. Consider elements $e^n$ and $x$. The distance between $e^n$ and $x$ is equal to one. It follows that the distance between $F(e^n)$ and $y$ should be at most one. We know that $F(e^n)$ is $te^{\sigma_n}$, where $t>0$ if $\alpha_n=1$ and $t<0$ if $\alpha_n=-1$. We know that $y_{\sigma_n}$ is either $1$ or $-1$. If $y_{\sigma_n}\ne\alpha_n$, then $|F(e^n)_{\sigma_n}-y_{\sigma_n}|>1$, which implies that the distance between $F(e^n)$ and $y$ is greater than one, but this can not be the case. Therefore, we must have $y_{\sigma_n}=\alpha_n$. The case $x_n=-1$ is analogous. Now, the fact that $y_{\sigma_n}=\alpha_n x_n$ for each $n\in\N$ and the fact that the sequences $x$ and $y$ are constant from some point imply that the sequence $\alpha$ should be also constant from some point.
    
    Now we want to show that for each $x\in B$ and $n\in\N$ the items 1)--3) are true. As $F$ is continuous and $B\setminus B_0$ is dense in $B$, then it suffices to consider the case where $\lim x_n\ne 0$. Note that we can not apply the approach used in the proof for $c_0$, because the subset $B^*_0$ is not dense in $B$. Denote by $S$ the set $\{n\in\N\colon x_n\ne 0\}$. For every $n\in S$ let $s^n=-\sgn(x_n)e^n$. Note that $B$ is covered by the ball $B(x,1)$ together with the balls $B(s^n,1)$, $n\in S$. It follows that $B$ should be also covered by the balls $B\big(F(x),1\big)$ and $B\big(F(s^n),1\big)$, $n\in S$.
    
    As with the case of $c_0$, the next step is to find an element $y\in B$, which does not belong to any of the balls $B\big(F(s^n),1\big)$, $n\in S$. The difficult part is to ensure that the sequence $y$ belongs to $B$. Let us construct a sequence $y$ as follows:
    
    \begin{itemize}
        \item for $k\not\in \sigma(S)$ let $y_k=\lim(\sgn(x_n)\alpha_n)$;
        \item for $n\in S$ let $y_{\sigma_n}=\sgn(x_n)\alpha_n$.
    \end{itemize}
    
    Let us ensure that the sequence $y$ belongs to $B$. Since $\lim x_n\ne 0$, then either $\lim x_n> 0$ or $\lim x_n < 0$. If $\lim x_n > 0$, then there exists an index $N$ such that $\sgn(x_n)=1$ for each $n\geq N$.
    If $\lim x_n < 0$, then there exists an index $N$ such that $\sgn(x_n)=-1$ for each $n\geq N$. Either way, the sequence $(\sgn x_n)$ is constant starting from index $N$. The sequence $(\alpha_n)$ is also constant from some point, as shown above. This means that the sequence $(\sgn(x_n)\alpha_n)$ is also constant from some point, so the limit $\lim(\sgn(x_n)\alpha_n)$ exists and is equal to $1$ or $-1$. Now we see that the sequence $y$ consists of ones and minus ones and is constant from some point. Therefore, $y$ belongs to $B$.
    
    Note that $y$ does not belong to any of the balls $B\big(F(s^n),1\big)$, $n\in S$. On the other hand, $y$ is an element of $B$ and $B$ should be covered by the balls $B\big(F(x),1\big)$ and $B\big(F(s^n),1\big)$, $n\in S$. It follows that $y$ should belong to $B\big(F(x),1\big)$. This implies items 2) and 3). The item 1) can be proved by the same argument as in the proof for $c_0$.
    
    Continuity of $F$ and the fact that $B\setminus B_0$ is dense in $B$ allow us to extend the result to all elements of $B$. Finally, we have to make sure $\sigma$ is a surjection. This can be proved by the same argument as in the proof for $c_0$.
\end{proof}

The previous lemma fixes a bijection $\sigma\colon\mathbb{N}\to\mathbb{N}$ and a sequence $\alpha\colon\mathbb{N}\to\{-1,1\}$, that is constant from some point. Let $\calA$ be the corresponding isometric automorphism of $c$ defined by $\calA(x)_{\sigma_n}=\alpha_n x_n$. Define $\tildeF$ as in the proof for $c_0$. Our goal is to show that $\tildeF$ is an identity map. The previous lemma implies the following properties of $\tildeF$.

\begin{lemma}\label{c_step1}
    For each $x\in B$ and $n\in\N$,
    \begin{itemize}
        \item if $x_n=0$, then $\tildeF(x)_n=0$;
        \item if $x_n<0$, then $\tildeF(x)_n\leq0$;
        \item if $x_n>0$, then $\tildeF(x)_n\geq0$.
    \end{itemize}
\end{lemma}

The three properties listed above imply the following properties of $\tildeF^{-1}$.

\begin{lemma}\label{c_step1a}
    For each $y\in B_0$ and $n\in\N$, if $y_n<0$, then $\tildeF^{-1}(y)_n<0$, and if $y_n>0$, then $\tildeF^{-1}(y)_n>0$.
\end{lemma}

In the proof for $c_0$, the next step was to prove Lemmas \ref{c0_step2} and \ref{c0_step3}. The proofs of these propositions work for the space $c$ as well. We only need to substitute $B$ for $B_0$. Therefore, we obtain the following.

\begin{lemma}\label{c_step2}
    For each $y\in B$ and $n\in\N$, if $y_n=1$ and $y_i\ne 0$ for all $i\ne n$, then $\tildeF^{-1}(y)_n=1$, and if $y_n=-1$ and $y_i\ne 0$ for all $i\ne n$, then $\tildeF^{-1}(y)_n=-1$.
\end{lemma}

\begin{lemma}\label{c_step3}
    For each $x\in B$ and $n\in\mathbb{N}$, if $x_n<0$, then $\tildeF(x)_n\in[x_n,0]$, and if $x_n>0$, then $\tildeF(x)_n\in[0,x_n]$.
\end{lemma}

As for the case of $c_0$, the previous lemma implies the following properties of the inverse function.

\begin{lemma}\label{c_step3a}
    For each $y\in B_0$ and $n\in\N$, if $y_n<0$, then $\tildeF^{-1}(y)_n\in[-1,y_n]$, and if $y_n>0$, then $\tildeF^{-1}(y)_n\in[y_n,1]$.
\end{lemma}

In the case of $c_0$, the next step was Lemma \ref{c0_step4}. The proof of this lemma relies on the fact that every element of $c_0$ attains its norm, but this is not true in $c$. Therefore, we are forced to use some alternative approach.

\begin{lemma}\label{c_step4}
    Let $x\in B$ be such that the set $S=\{i\in\N\colon x_i\not\in\{-1,1\}\}$ is finite. Then $\tildeF(x)=x$.
\end{lemma}
\begin{proof}
    Let us proceed by induction on the number of elements of $S$. For the base of induction, consider the case where the set $S$ is empty. If the set $S$ is empty, then $x_i\in\{-1,1\}$ for every $i\in\N$ and the application of Lemma \ref{c_step3a} gives $\tildeF^{-1}(x)=x$, which is equivalent to $\tildeF(x)=x$. This proves the base of induction.
    
    Now, let $N$ be an arbitrary non-negative integer. Suppose that the claim holds whenever the set $S$ has at most $N$ elements. Let us prove that the claim also holds when the set $S$ has $N+1$ elements. Suppose that the set $S$ has $N+1$ elements. We need to show $\tildeF(x)=x$, which is equivalent to showing $\tildeF^{-1}(x)=x$. Applying Lemma \ref{c_step3a}, we obtain that $\tildeF^{-1}(x)_n=x_n$ for every $n\not\in S$. It remains to show that $\tildeF^{-1}(x)_n=x_n$ is true for every $n\in S$. Let $n$ be an arbitrary element of $S$. Consider sequences $y$ and $z$ defined by $y_n=1$, $z_n=-1$ and $y_i=z_i=\tildeF^{-1}(x)_i$ for every $i\ne n$. Note that the sets $S_y=\{i\in\N\colon y_i\not\in\{-1,1\}\}$ and $S_z=\{i\in\N\colon z_i\not\in\{-1,1\}\}$ are contained in the set $S\setminus\{n\}$. This implies that the sets $S_y$ and $S_z$ have at most $N$ elements. Therefore, we can apply the induction hypothesis to obtain $\tildeF(y)=y$ and $\tildeF(z)=z$. Since the sequences $y$, $z$ and $\tildeF^{-1}(x)$ coincide for all indices distinct from $n$, then the distance between elements $\tildeF^{-1}(x)$ and $y$ is equal to $|\tildeF^{-1}(x)_n-y_n|$ and the distance between elements $\tildeF^{-1}(x)$ and $z$ is equal to $|\tildeF^{-1}(x)_n-z_n|$. It follows that the distance between elements $x$ and $y$ is at most $|\tildeF^{-1}(x)_n-y_n|$ and the distance between elements $x$ and $z$ is at most $|\tildeF^{-1}(x)_n-z_n|$. Combining these two facts, we obtain $\tildeF^{-1}(x)_n=x_n$.
\end{proof}

From the last lemma it follows that $\tildeF$ is an identity map on $B^*_1\cup B^*_{-1}$. Since $\tildeF$ is continuous and $B^*_1\cup B^*_{-1}$ is dense in $B_1\cup B_{-1}$, then $\tildeF$ is also an identity map on $B_1\cup B_{-1}$. It turns out that we can say something about other levels too. 

\begin{lemma}\label{c_step5}
    Let $x\in B$, $h=\lim x_k$ and $n\in\N$.
    \begin{enumerate}[1)]
        \item If $|x_n|<|h|$, then $\tildeF(x)_n=x_n$.
        \item If $x_n\geq |h|$, then $\tildeF(x)_n\in[|h|,x_n]$.
        \item If $x_n\leq -|h|$, then $\tildeF(x)_n\in[x_n,-|h|]$.
    \end{enumerate}
\end{lemma}
\begin{proof}
    For the case $h=0$, the three items follow from Lemma  \ref{c_step3} and the item 1) of Lemma \ref{c_step1}, so it remains to consider the case $h\ne 0$.
    
    Let us start with proving the first item. For the case $x_n=0$, the claim follows from Lemma \ref{c_step1}, so it remains to consider the case $x_n\ne 0$. Let $\varepsilon = |h|-|x_n|$. Note that $\varepsilon>0$. Since the sequence $x$ converges to $h$, then there exists an index $N\in\N$ such that $|x_k-h|<\varepsilon$ for each $k\geq N$. Note that $n<N$. Define a sequence $y\in B$ as
    
    \[
        y_k=
        \begin{cases}
            \sgn(x_n),   & k=n,\\
            x_k, & k<N,\ k\ne n,\\
            \sgn(h), & k\geq N.
        \end{cases}
    \]
    
    Note that the sequence $y$ satisfies the conditions of Lemma \ref{c_step4}, hence we have $\tildeF(y)=y$. Compare sequences $x$ and $y$. For $k=n$ we have $|x_k-y_k|=1-|x_n|$. For $k<N$, $k\ne n$ we have $|x_k-y_k|=0$. For $k\geq N$ we have $|x_k-y_k|<1-|x_n|$. It follows that the distance between $x$ and $y$ is equal to $1-|x_n|$. Therefore, the distance between $\tildeF(x)$ and $y$ is at most $1-|x_n|$ (recall that $\tildeF(y)=y$). If $x_n<0$, then the latter fact implies $\tildeF(x)_n\leq x_n$, while Lemma \ref{c_step3} implies $\tildeF(x)_n\geq x_n$. If $x_n>0$, then the latter fact implies $\tildeF(x)_n\geq x_n$, while Lemma \ref{c_step3} implies $\tildeF(x)_n\leq x_n$. In either case we have $\tildeF(x)_n=x_n$ as wanted.
    
    Now, let us consider the second item. By Lemma \ref{c_step3} we know $\tildeF(x)_n\leq x_n$, so it remains to show $\tildeF(x)_n\geq |h|$. Let $\varepsilon$ be an arbitrary positive number. Since the sequence $x$ converges to $h$, then there exists an index $N\in\N$ such that $|x_k-h|<\varepsilon$ for each $k\geq N$. If it happens that $N\leq n$, then choose $N$ to be any index greater than $n$. Define a sequence $y\in B$ as before. Note that the sequence $y$ satisfies the conditions of Lemma \ref{c_step4}, hence we have $\tildeF(y)=y$. Compare sequences $x$ and $y$. For $k=n$ we have $|x_k-y_k|=1-|x_n|\leq 1-|h|$. For $k<N$, $k\ne n$ we have $|x_k-y_k|=0$. For $k\geq N$ we have $|x_k-y_k|<1-|h|+\varepsilon$. It follows that the distance between $x$ and $y$ is at most $1-|h|+\varepsilon$. Therefore, the distance between $\tildeF(x)$ and $y$ is also at most $1-|h|+\varepsilon$ (recall that $\tildeF(y)=y$). This yields $\tildeF(x)_n\geq |h|-\varepsilon$. Since $\varepsilon$ was arbitrary, then it follows that $\tildeF(x)_n\geq |h|$. This concludes the proof of the second item. The proof of the third item is analogous. 
\end{proof}

We can make some conclusions from the properties obtained in the last lemma. First, we see that $\tildeF$ preserves the limit -- for every $x\in B$ we have $\lim \tildeF(x)_k=\lim x_k$. Moreover, we see that the inverse function has the following property.

\begin{lemma}\label{c_step5a}
    Let $y\in B$, $h=\lim y_k$ and $n\in\N$. If $|y_n|<|h|$, then $\tildeF^{-1}(y)_n=y_n$.
\end{lemma}
\begin{proof}
    Denote $\tildeF^{-1}(y)$ by $x$. As mentioned above, $\tildeF$ preserves the limit. Therefore, we have $\lim x_k=\lim y_k=h$. We can have three cases: $|x_n|<|h|$, $x_n\geq |h|$ and $x_n\leq -|h|$. If $x_n\geq |h|$, then Lemma \ref{c_step5} implies that $y_n\geq |h|$, which contradicts our assumption. If $x_n\leq -|h|$, then Lemma \ref{c_step5} implies that $y_n\leq -|h|$, which contradicts our assumption. Therefore, we are left with the case $|x_n|<|h|$, so Lemma \ref{c_step5} implies that $y_n=x_n$.
\end{proof}

We are almost done. Recall that in the case of $c_0$ the last step was to show that the set $B^S_0$, where $S$ is a finite subset of $\N$, is mapped bijectively onto itself. To finish the proof at hand, it will suffice to show the same for the set $B^S_h$, where $S$ is a finite subset of $\N$ and $h\in[-1,1]$. Lemma \ref{c_step5} implies the inclusion $\tildeF(B^S_h)\subset B^S_h$, so it remains to show that the same is true for the inverse function. This is exactly what the next lemma asserts. It will be more convenient to limit ourselves to the case $h\in(-1,1)\setminus\{0\}$. Fortunately, this will be sufficient.

\begin{lemma}\label{c_step6}
    Let $h\in(-1,1)\setminus\{0\}$ and let $S$ be a finite subset of $\N$. Then $\tildeF^{-1}(B^S_h)\subset B^S_h$.
\end{lemma}
\begin{proof}
    Let us consider the case $h>0$. Let $y$ be an arbitrary element of the set $B^S_h$. Our goal is to show that $\tildeF^{-1}(y)\in B^S_h$. To prove this, we need to show that $\tildeF^{-1}(y)_n=h$ for every $n\in\N\setminus S$. Let $n$ be an arbitrary element of $n\in\N\setminus S$. Since $y_n=h$ and $h>0$, then Lemma \ref{c_step3a} says $\tildeF^{-1}(y)_n\geq h$, so we only need to show that the reverse inequality is also true. For the sake of contradiction, suppose that $\tildeF^{-1}(y)_n > h$. Denote $1-h$ by $d$. Define a sequence $z\in B$ as 
    
    \[
        z_k=
        \begin{cases}
            1,              & k=n,\\
            y_k,            & |y_k|<h,\\
            h+d/2(1-1/2^k),    & y_k\geq h,\ k\ne n,\\
            -h-d/4, & y_k\leq -h.
        \end{cases}
    \]
    
    It is straightforward to check that $z_k\in[-1,1]$ for every $k\in\N$. We see that for every $k\not\in S\cup\{n\}$ we have the third case. Since the set $S\cup\{n\}$ is finite and $h+d/2(1-1/2^k)\to h+d/2$, then we also have $z_k\to h+d/2$. Therefore, we see that $z$ is indeed an element of $B$.
    
    Denote $\tildeF^{-1}(y)$ and $\tildeF^{-1}(z)$ by $y'$ and $z'$. To obtain a contradiction, let us show that the distance between $y'$ and $z'$ is smaller than $d$. Since $z_n=1$, then Lemma \ref{c_step3a} says $z'_n=1$. Consider an index $k\ne n$. Recall that $\lim z_k=h+d/2$. Since we have $|z_k|< |h+d/2|$, then Lemma \ref{c_step5a} implies $z'_k=z_k$. It follows that $z'=z$. Therefore, we need to show that the distance between $y'$ and $z$ is smaller than $d$. According to our assumption, we have $y'_n>h$, hence $|y'_n-z_n|<d$. To show that the distance between $y'$ and $z$ is smaller than $d$, it will suffice to show that $|y'_k-z_k|\leq 3d/4$ for every $k\ne n$. Consider the case $|y_k|<h$. Lemma \ref{c_step5a} implies $y'_k=y_k$ and the definition of $z$ implies $z_k=y_k$, so $y'_k=z_k$ and $|y'_k-z_k|=0$. Let us consider the cases $y_k\geq h$ and $y_k\leq -h$. If $y_k\geq h$, then Lemma \ref{c_step3a} implies that $y'_k\in[h,1]$, and if $y_k\leq -h$, then Lemma \ref{c_step3a} implies that $y'_k\in[-1,-h]$. In either case, the greatest possible distance between $y'_k$ and $z_k$ is $3d/4$. This shows that the distance between $y'$ and $z$ is smaller than $d$. Since the distance between $y'$ and $z$ is smaller than $d$, then the distance between $y$ and $z$ should be also smaller than $d$, but we have $|y_n-z_n|=d$, which is a contradiction. The case $h<0$ is analogous.
\end{proof}

Now, we can finish the proof. The remaining part is very similar to the way we finished the proof for $c_0$.

\begin{theorem}\label{c_main_result}
    $F$ is an isometry.
\end{theorem}
\begin{proof}
    Let $h\in(-1,1)\setminus\{0\}$ and let $S$ be a finite subset of $\N$. If $x\in B^S_h$, then $\lim x_k=h$ and Lemma \ref{c_step5} implies that $\tildeF(x)_n=h$ whenever $x_n=h$. This yields the inclusion $\tildeF(B^S_h)\subset B^S_h$. Applying Lemma \ref{c_step6}, we obtain the inclusion $\tildeF^{-1}(B^S_h)\subset B^S_h$. Combining these two together, we see that the set $B^S_h$ is mapped bijectively onto itself. It follows that the restriction of $\tildeF$ to $B^S_h$ is a non-expansive bijection from the unit ball of a finite-dimensional space onto itself (the set $B^S_h$, as a metric space, can be identified with $B^S_0$). Since the unit ball of a finite-dimensional space is plastic, then it follows that the restriction of $\tildeF$ to $B^S_h$ is an isometry. Theorem \ref{Mankiewicz} says that the latter is actually a restriction of an isometric automorphism of the underlying finite-dimensional space. Combining this with some previously acquired information, we can conclude that the restriction of $\tildeF$ to $B^S_h$ is an identity map.
    
    Since $\tildeF$ is an identity map on $B^S_h$ for every finite subset $S$, then the restriction of $\tildeF$ to $B^*_h$ is also an identity map, because the latter is the union of all the subsets $B^S_h$, where $S$ is a finite subset of $\N$. Since $\tildeF$ is an identity map on $B^*_h$, $\tildeF$ is continuous and $B^*_h$ is dense in $B_h$, then it follows that the restriction of $\tildeF$ to $B_h$ is also an identity map.
    
    We have seen that $\tildeF$ is an identity map on $B_h$ for every $h\in(-1,1)\setminus\{0\}$. Previously, we have also seen that $\tildeF$ is an identity map on $B_{-1}$ and $B_1$. It follows that the restriction of $\tildeF$ to $B\setminus B_0$ is an identity map. Since $\tildeF$ is continuous and $B\setminus B_0$ is dense in $B$, then it follows that $\tildeF$ is an identity map. This means that $F$ is a restriction of an isometric automorphism of $c$ defined by $\calA(x)_{\sigma_n}=\alpha_n x_n$. In particular, $F$ is an isometry.
\end{proof}

\section*{Acknowledgements}
The original results presented in this paper are part of the author's bachelor's thesis \textquote{Plasticity of the unit ball of a Banach space}, defended at the University of Tartu on 10 June 2021 and supervised by Rainis Haller (University of Tartu) and Olesia Zavarzina (V.~N.~Karazin Kharkiv National University). The author wishes to thank Vladimir Kadets for moral support and Aleksei Lissitsin for pointing out some typos.
\newpage
\printbibliography[title={References}]

\end{document}